\documentclass[11pt]{article}
\usepackage[T1]{fontenc}
\usepackage[utf8]{inputenc}
\usepackage{lmodern,amsmath,amssymb,amsthm,mathtools,booktabs,array,longtable}
\usepackage[margin=0.9in]{geometry}
\usepackage{microtype,enumitem,fancyhdr}
\usepackage[hidelinks]{hyperref}
\hypersetup{pdfauthor={Navid Rezazadeh},pdftitle={Matroid-Aware Spectral Certificates for Multiplicative Coverage Approximation},pdfsubject={Formal research manuscript for independent mathematical review},pdfkeywords={matroid, coverage, approximation, research draft}}
\setlist{itemsep=2pt,topsep=4pt}
\newtheorem{theorem}{Theorem}[section]
\newtheorem{lemma}[theorem]{Lemma}
\newtheorem{proposition}[theorem]{Proposition}
\newtheorem{corollary}[theorem]{Corollary}
\theoremstyle{remark}
\newcommand{\R}{\mathbb R}\newcommand{\Q}{\mathbb Q}
\newcommand{\one}{\mathbf1}\newcommand{\OPT}{\operatorname{OPT}}
\newcommand{\diag}{\operatorname{diag}}\newcommand{\conv}{\operatorname{conv}}
\newcommand{\rank}{\operatorname{rank}}

\numberwithin{equation}{section}
\title{Matroid-Aware Spectral Certificates for\\Multiplicative Coverage Approximation}
\author{Navid Rezazadeh\\
  {\small UC Irvine}\\
  {\small\href{mailto:nrezazad@uci.edu}{\texttt{nrezazad@uci.edu}}}}
\date{12 October 2026}
\begin{document}
\maketitle
\vspace{-1.2em}
\begin{abstract}
We maximize the weight of edges covered by selected vertices, together with nonnegative individual rewards, under an arbitrary matroid constraint. The continuous coverage objective is quadratic and need not be concave. Our method searches projections onto a small set of directions and solves a concave subproblem for each projection center. A coverage-specific inequality converts the subproblem bounds into a multiplicative guarantee, while matroid rounding restores an integral solution. Let $A$ be the weighted adjacency matrix and $D$ the diagonal matrix of singleton values, restricted to positive-value elements after objective-preserving preprocessing. Set $\varrho=\lambda_{\max}(D^{-1/2}AD^{-1/2})$. For an integer budget $q\ge0$, let $\tau_q$ be the smallest nonnegative threshold for which $A+\tau_qD$ has at most $q$ negative eigenvalues on the span of differences of matroid bases. We obtain the factor $1-\varrho\tau_q/[2(\varrho+\tau_q)]-\varepsilon$ in time $(2+q/\varepsilon)^{O(q+1)}N^{O(1)}$, where $N$ is the binary input length and a zero-over-zero loss is zero. The algorithm returns a rational upper bound on the integral optimum. A separate exact-algebraic construction produces polynomial-bit low-rank rational corrections with positive threshold slack; its proof isolates eigenvalues, keeps eigenspace bases unnormalized, and approximates scalar factors only at the final step. The numerical proposal prototype is not used to justify this complexity statement. Supplied certificates also yield robustness to locally bounded edge additions. These are threshold-dependent coverage guarantees, not a new general spectral-search principle or a universal approximation improvement.
\end{abstract}

\section{Introduction and prior work}
Selecting a vertex covers every edge incident to it. We seek a matroid-independent set that maximizes the total covered weight together with nonnegative individual vertex rewards. Write $f(S)$ for this value and $\OPT$ for its maximum. The matroid may impose dependencies across the entire graph; a cardinality constraint is only one special case. The general coverage baseline is a $3/4$ approximation~\cite{HS23}. We ask when the interaction between the graph and the matroid permits a stronger guarantee.

The multilinear coverage value $F(x)$ is the expected value when vertices are selected independently with marginal probabilities $x$. It is quadratic. We optimize it on the matroid base polytope, the convex hull of maximal independent-set indicator vectors. Each point in this polytope can be rounded to a basis without decreasing $F$, so optimizing $F$ is sufficient. The difficulty is that $F$ need not be concave: its negative adjacency term can have positive curvature along some directions.

Our algorithm corrects that curvature by adding a diagonal variance term and subtracting penalties along a small number of directions. It searches a grid of penalty centers and solves a concave problem for each center. For each optimal basis, the grid contains a branch that preserves its value. A coverage-specific inequality then converts the branch upper bound into a lower bound on the rounded solution. The grid is scaled by the value of a feasible basis, not by total graph weight, so the exponential part of the search does not depend on the number of vertices.

Only directions between feasible base-polytope points matter for concavity. Restricting the curvature test to this space can remove directions that the graph alone would suggest correcting. We compute the space through matroid exchanges and determine the required correction rank by counting negative eigenvalues. With a small threshold allowance, the required rational correction can also be constructed in polynomial bit complexity. Appendix~\ref{app:rational-correction} proves this last step by exact univariate algebra; it does not infer factor extraction from an approximate full matrix or from the behavior of a numerical prototype.

\subsection{Result preview and computational scope}
Let $A$ be the weighted adjacency matrix and let $D$ contain the singleton values $f(\{i\})$ on its diagonal, after removing unselectable and zero-value elements without changing the objective. Set $\varrho=\lambda_{\max}(D^{-1/2}AD^{-1/2})$. For an integer budget $q\ge0$, let $\tau_q$ be the smallest nonnegative threshold for which $A+\tau_qD$ has at most $q$ negative directions within the span of differences of matroid bases. 

Theorem~\ref{thm:intrinsic-main} gives the factor $1-\varrho\tau_q/[2(\varrho+\tau_q)]-\varepsilon$ in time $(2+q/\varepsilon)^{O(q+1)}N^{O(1)}$, where $N$ is input bit length and $\varepsilon$ is the requested ratio loss. Theorem~\ref{thm:joint} first establishes the guarantee for a supplied rational correction. Lemma~\ref{lem:rational-correction} then supplies that correction algorithmically, with its complete factor and bit-complexity argument in Appendix~\ref{app:rational-correction}. The modular zero case is solved exactly.

The exponent of $N$ does not depend on $q$, but the running time may be exponential in $q$. The method also computes an upper bound on $\OPT$. Its exact-algebraic factor construction uses standard rational polynomial factorization and real-root isolation. It does not require a common algebraic field containing all eigenvalues. The older numerical factor generator still accepts outputs only after exact checks; its finite precision guard is not part of the algorithm used in the complexity proof.

\subsection{What is established and what requires comparison}
The search mechanism is not a new general optimization principle. Vavasis~\cite{Vavasis92} already gave approximation schemes with complexity parameterized by the number of negative eigenvalues. Anderson, Hopkins, Rajaraman, and Steurer~\cite{AHRS25} explicitly discuss quadratic optimization over general convex domains with efficient separation oracles. 

Cevallos, Eisenbrand, and Zenklusen~\cite{CEZ16} use concavity on slices and matroid rounding for negative-type diversity. Soltanalian and Mousavi~\cite{SM26} discuss positive-semidefinite low-rank maximization, including matroid bases, and carefully distinguish that setting from indefinite optimization.

Rezazadeh and Kia~\cite{RezazadehKia23} use multilinear extensions and stochastic pipage rounding for distributed submodular maximization under partition-matroid constraints. Their privacy-focused work~\cite{RezazadehKia22} studies the tradeoff between approximation performance and privacy in the same distributed-constraint setting. These methods address networked strategy selection, rather than the matroid-aware spectral certificates analyzed here.

For coverage itself, Huang and Sellier~\cite{HS23} record the general $3/4$ approximation. The preprint of Paschos~\cite{Paschos19} presents a polynomial-time approximation scheme (PTAS) for weighted bipartite maximum $k$-vertex cover. Thus, the uniform-matroid specialization is not an appropriate novelty claim. The possible contribution here is a particular relative-to-$\OPT$ guarantee for arbitrary matroids, its matrix certificate, and its stability under changes to the graph. A complete proof-level historical priority determination remains necessary; no novelty claim is implied by presenting the full derivation.

\paragraph{A scale distinction.}
An additive error proportional to total graph weight does not automatically give a fixed relative error for matroid coverage. Let $W$ denote total edge weight. On an unweighted $d$-regular graph with $n$ vertices with the uniform rank-one matroid, $\OPT=d$ and $W=nd/2$. Therefore $\OPT/W=2/n$. An error of $\epsilon W$ is $(\epsilon n/2)\OPT$. This observation does not prove that prior algorithms cannot be adapted. It identifies the conversion that an attribution comparison must check.

The argument below follows the dependencies of the final guarantee. Sections~\ref{sec:problem} and~\ref{sec:conversion} define the certificate and its coverage conversion. Sections~\ref{sec:fpt} and~\ref{sec:spectral} improve the search scale and ratio. Section~\ref{sec:intrinsic} characterizes the correction rank and constructs a rational correction automatically; Appendix~\ref{app:rational-correction} supplies the exact-algebraic proof. Applications and comparisons show what the resulting parameters do and do not distinguish. Detailed branch optimization and bit control are in Appendix~\ref{sec:algorithm}.

\section{Problem and the supplied-certificate guarantee}\label{sec:problem}
We first specify a certificate that makes each optimization branch concave. This establishes a baseline theorem with explicit input conditions. Later sections improve its ratio and search bound, characterize the necessary correction rank, and obtain the automatic construction.

Let $V$ have $n$ elements. Let $M=(V,\mathcal I)$ be a matroid supplied by an independence oracle. A basis is a maximal independent set; all bases have the same cardinality $r$. Write $\mathcal B$ for the family of bases and
\[
 B(M)=\conv\{\one_B:B\text{ is a basis}\}.
\]
Here $\one_B$ is the indicator vector of $B$.

A matroid loop is an element whose singleton is not independent. For the initial theorem assume there are no matroid loops. The reduction in Section~\ref{sec:loops}, immediately below, removes this assumption while retaining rewards on incident edges. The graph itself has no self-loops. Parallel edges are combined, and zero-weight edges may be omitted.

The objective is
\begin{equation}\label{eq:obj}
 f(S)=\sum_{i\in S}\ell_i+
 \sum_{ij\in E}w_{ij}\one[S\cap\{i,j\}\ne\varnothing],
 \qquad \OPT=\max_{S\in\mathcal I}f(S),
\end{equation}
where $\one[\cdot]$ is the indicator of the stated event and all weights and scores are nonnegative rationals. Every independent set extends to a basis. Hence a basis attains $\OPT$.

Let $A$ be the symmetric weighted adjacency matrix with zero diagonal. Set
\[
 b_i=\ell_i+\sum_{j:ij\in E}w_{ij},\quad
 D=\diag(b),\quad S_b=\sum_i b_i,\quad B_0=\max_i b_i.
\]
Thus $b_i=f(\{i\})$ is the singleton value. If $S_b=0$, every feasible set has value zero. We henceforth assume $S_b>0$. Since every singleton is feasible, looplessness gives
\begin{equation}\label{eq:scale}
 B_0\le\OPT,\qquad S_b\le nB_0.
\end{equation}

Three continuous functions play different roles. The multilinear objective $F$ will be preserved by rounding; the concave cap $L$ bounds branch values; and the variance term $T$ vanishes on integral vectors. Define them on $[0,1]^V$ by
\begin{align}
 F(x)&=b^\top x-\tfrac12x^\top Ax,\label{eq:F}\\
 L(x)&=\ell^\top x+\sum_{ij\in E}w_{ij}\min(1,x_i+x_j),\label{eq:L}\\
 T(x)&=\tfrac12\sum_i b_i x_i(1-x_i).\label{eq:T}
\end{align}
At a zero-one vector, $F=L=f$ and $T=0$.

\begin{table}[ht]\centering\small
\begin{tabular}{@{}p{0.17\linewidth}p{0.76\linewidth}@{}}\toprule
Symbol & Meaning\\\midrule
$n,r$ & Ground-set size and matroid rank\\
$A,b,D$ & Adjacency matrix, singleton values, and their diagonal matrix\\
$\tau$ & Rational diagonal-curvature parameter in $[0,1]$\\
$P$ & Supplied positive-semidefinite correction\\
$q$ & Number of supplied rank-one correction terms; $\rank(P)\le q$\\
$\kappa_j,c_j$ & Rational factor weights $\kappa_j>0$ and vectors $c_j\in\Q^V$\\
$\alpha(\tau)$ & $(2+\tau)/(2+2\tau)$\\
$\varepsilon,\zeta$ & Ratio accuracy and internal tolerance $\zeta=\varepsilon B_0/4$\\
$\overline U$ & Computed upper bound on the integral optimum\\\bottomrule
\end{tabular}
\caption{Parameters of the supplied-certificate theorem. Here $q$ counts correction terms; it is distinct from the matroid rank $r$.}
\end{table}

A symmetric matrix $Q\succeq0$ has a nonnegative quadratic form on every real vector. In the supplied factorization below, $\kappa_j>0$ and $c_j\in\Q^V$. The correction $P$ removes the nonconcavity left after adding the diagonal term $\tau D$. Write $\alpha(\tau)=(2+\tau)/(2+2\tau)$ and let $q$ be the number of supplied correction terms. The following initial guarantee is the platform for the stronger theorem, not the final complexity claim.

\begin{theorem}[Preliminary supplied-certificate approximation]\label{thm:master}
Suppose the input includes rational data $\tau\in[0,1]$ and
\begin{equation}\label{eq:cert}
 P=\sum_{j=1}^q\kappa_jc_jc_j^\top,
 \qquad 0\preceq P\preceq D,\qquad A+\tau D+P\succeq0.
\end{equation}
For rational $0<\varepsilon<3/4$, there is a deterministic algorithm returning a basis $S$ and a rational $\overline U$ with
\begin{equation}\label{eq:main}
 \OPT\le\overline U,\qquad
 f(S)\ge\bigl(\alpha(\tau)-\varepsilon\bigr)\overline U
 \ge\bigl(\alpha(\tau)-\varepsilon\bigr)\OPT.
\end{equation}
The number of optimization branches is at most
\begin{equation}\label{eq:branches}
 \left(2+\sqrt{\frac{qn}{2\varepsilon}}\right)^q
\end{equation}
when $q\ge1$, and is one when $q=0$. The remaining work is polynomial in this branch count, the rational input encoding, and $1/\varepsilon$. In particular, the algorithm is polynomial-time for fixed $q$.
\end{theorem}
The certificate can be checked by exact rational positive-semidefiniteness tests. The theorem does not claim an algorithm for finding a minimum-rank correction. Its stated dependence is not a fixed-parameter running time of the form $g(q)n^{O(1)}$.

\subsection{Objective-preserving preprocessing}\label{sec:loops}
An unselectable vertex cannot be discarded together with all its incident rewards. The next reduction transfers each such reward to the endpoint that can still be selected, preserving both the objective and the supplied certificate.
\begin{proposition}[Objective-preserving loop deletion]\label{prop:loops}
Theorem~\ref{thm:master} applies to an arbitrary matroid after a polynomial-time reduction. The supplied certificate restricts to the surviving coordinates, and the number of correction terms does not increase.
\end{proposition}
\begin{proof}
Find all matroid loops by singleton queries, and let $V'$ be the other elements. Delete the loops from the matroid and from the graph. For each surviving vertex replace its modular score by
\[
 \ell'_i=\ell_i+\sum_{j\notin V'}w_{ij}.
\]
An edge joining a loop to a surviving vertex is covered exactly when that surviving vertex is selected. Edges between loops are never covered. Thus the objective on every feasible selection is preserved exactly. 

Moreover, the new singleton value is $b'_i=b_i$, not the smaller degree in the induced graph. Therefore $A'$ and $D'$ are principal submatrices of $A$ and $D$. Restrict each $c_j$ to $V'$. Positive semidefiniteness and $P\preceq D$ are preserved by principal restriction. 

The reduced matroid is loopless. If no reward survives, any basis is optimal with value zero; otherwise \eqref{eq:scale} applies in the reduced ground set.
\end{proof}
This proof does not claim that deleting graph vertices preserves a normalized spectral gap. It preserves a different object: the rational matrix certificate together with modular scores.

\section{Concave branches and the coverage conversion}\label{sec:conversion}
Two issues must be resolved before optimizing the branches: their objective must imply a rounded coverage value, and their collection must retain an upper bound on the true optimum. The scalar inequality addresses the first issue; the projection grid addresses the second.

For the baseline construction set $\zeta=\varepsilon B_0/4$. Section~\ref{sec:fpt} later replaces this scale by a feasible basis value. The variance term added for concavity has a bounded cost relative to coverage. This edgewise estimate supplies the conversion used in every branch.

\begin{lemma}[Coverage-specific scalar bound]\label{lem:scalar}
For all $x\in[0,1]^V$, $T(x)\le F(x)-L(x)/2$.
\end{lemma}
\begin{proof}
For an edge put $a=x_i$, $b=x_j$, and $y=a+b-\min(1,a+b)\ge0$. Its contribution to $F-L/2-T$ is
\[
 w_{ij}\left[a+b-ab-\tfrac12\min(1,a+b)
 -\tfrac12\{a(1-a)+b(1-b)\}\right]
 =\tfrac12w_{ij}\bigl((a-b)^2+y\bigr)\ge0.
\]
A modular score contributes $\ell_i x_i^2/2\ge0$. Sum these contributions.
\end{proof}

Let $a\in\R^q$ be a vector of target values for the projections $c_j^\top x$. Define
\begin{align}
 G_a(x)&=F(x)+\tau T(x)
 -\tfrac12\sum_{j=1}^q\kappa_j(c_j^\top x-a_j)^2,\label{eq:G}\\
 \Phi_a(x)&=\min\{L(x),G_a(x)+\zeta\}.\label{eq:Phi}
\end{align}

Since $\nabla^2G_a=-(A+\tau D+P)$, $G_a$ is concave. The function $L$ is concave. The pointwise minimum of concave functions is concave, as its hypograph is the intersection of their hypographs. Hence $\Phi_a$ is concave.

The cap prevents the corrected objective from exceeding the ordinary fractional coverage expression. Combining that cap with the scalar estimate gives a multiplicative relation, even when a particular branch value is negative.

\begin{lemma}[Multiplicative conversion]\label{lem:conversion}
For all $x\in[0,1]^V$ and all centers $a$,
\begin{equation}\label{eq:conversion}
 F(x)\ge\alpha(\tau)\Phi_a(x)-\frac{\zeta}{1+\tau}.
\end{equation}
\end{lemma}
\begin{proof}
Put $t=\Phi_a(x)$. Drop the nonnegative penalty in $G_a$, and use $L(x)\ge t$ and Lemma~\ref{lem:scalar}:
\[
 t\le F+\tau T+\zeta
 \le (1+\tau)F-\frac\tau2L+\zeta
 \le (1+\tau)F-\frac\tau2t+\zeta.
\]
Rearrangement proves the claim. No sign assumption on $t$ is needed.
\end{proof}

For every correction coordinate compute
\[
 l_j=\min_{x\in B(M)}c_j^\top x,\quad
 u_j=\max_{x\in B(M)}c_j^\top x,\quad w_j=u_j-l_j.
\]
Two greedy matroid optimizations find these rational endpoints. Greedy optimization here is over bases, so negative coefficients are allowed.

The penalty can lower the value of a basis unless a grid center is close to its projection. Bounding each projection width permits a finite grid that preserves every basis in at least one branch.

\begin{lemma}[Projection widths and a covering grid]\label{lem:grid}
For every $j$, $\kappa_jw_j^2\le S_b$. There is a rational product grid $\mathcal A$ with the branch bound in Theorem~\ref{thm:master} such that every basis $B$ has a center $a\in\mathcal A$ with $\Phi_a(\one_B)=f(B)$.
\end{lemma}
\begin{proof}
Each rank-one term is bounded above by $P$, hence by $D$. For two points attaining the interval endpoints, their difference $z$ has coordinates in $[-1,1]$. Thus
\[
 \kappa_jw_j^2=z^\top\kappa_jc_jc_j^\top z\le z^\top Dz\le S_b.
\]

If $w_j=0$, use the single center $l_j$. Otherwise divide $[l_j,u_j]$ into
\[
 N_j=\max\left\{1,\left\lceil
 \sqrt{\frac{q\kappa_jw_j^2}{8\zeta}}\right\rceil\right\}
\]
equal subintervals, and use all endpoints. For any basis projection a nearest center has coordinate error at most $w_j/(2N_j)$. Its total quadratic penalty is at most
\[
 \sum_{j:w_j>0}\frac{\kappa_jw_j^2}{8N_j^2}\le\zeta.
\]

At a basis $T=0$ and $L=F=f$, so $G_a+\zeta\ge f$ and the cap equals $f$. Finally $N_j+1\le2+\sqrt{qS_b/(8\zeta)}$. Substitute $\zeta=\varepsilon B_0/4$ and \eqref{eq:scale}. For $q=0$ there is one empty-tuple center and no penalty.
\end{proof}
Consequently,
\begin{equation}\label{eq:global-upper}
 \OPT\le\max_{a\in\mathcal A}\max_{x\in B(M)}\Phi_a(x).
\end{equation}
One branch need not upper-bound all bases. Only the maximum across the grid has this property.

\subsection{How a branch becomes a certified solution}
Appendix~\ref{sec:algorithm} gives the complete rational optimizer. It replaces each minimum by a smooth concave lower approximation, uses a greedy linear-optimization gap to certify the branch optimum, and preserves an explicit convex combination of bases. Deterministic exchange rounding then returns a basis without decreasing $F$. The appendix also proves Theorem~\ref{thm:master}, including all smoothing, stopping, and encoding losses.

\section{Feasible-value normalization and fixed-parameter search}\label{sec:fpt}
The grid becomes expensive if its accuracy is scaled by a single vertex or by the whole ground set. A maximum-singleton-weight basis supplies a feasible objective scale. Bounding every projection width by that scale removes the ground-set size from the exponential search factor.

For a vertex-indexed vector $z$, write $z(S)=\sum_{i\in S}z_i$. Write $d_i=\sum_j w_{ij}$, $\Lambda=\max_{B\in\mathcal B}b(B)$, and let $B_b$ be a maximizing basis obtained by weighted matroid greedy. Put $L_0=f(B_b)$. In all nontrivial cases $L_0>0$.
Each internal edge is counted twice in the sum of singleton values. This gives a lower bound on the value of the greedy basis as well as an upper bound on the unknown optimum.

\begin{lemma}[Feasible objective normalization]\label{lem:scale-improved}
The following inequalities hold:
\begin{equation}\label{eq:scale-improved}
\frac12\Lambda\le L_0\le\OPT\le\Lambda,
\qquad \sum_i b_i\le 2nL_0.
\end{equation}
If $c=\max_{i:b_i>0}d_i/b_i$ is the total curvature, then in fact $L_0\ge(1-c/2)\Lambda$.
\end{lemma}
\begin{proof}
For a vertex set $S$, let $E[S]$ be the edges with both endpoints in $S$. Then
\[
f(S)=b(S)-\sum_{ij\in E[S]}w_{ij}.
\]
Every edge internal to $S$ occurs twice in $\sum_{i\in S}d_i$. Therefore
$\sum_{ij\in E[S]}w_{ij}\le\frac12d(S)\le\frac c2 b(S)$.
Substitute $S=B_b$, note $B_b$ is feasible, and use $f(S)\le b(S)\le\Lambda$ for every independent $S$ (extend $S$ to a basis). For the final inequality, looplessness means each vertex extends to a basis, so $b_i\le\Lambda\le2L_0$; sum over $i$.
\end{proof}

Two bases attain the endpoints of each projection interval. Their total singleton weight bounds the weighted squared width $\kappa_jw_j^2$, so no sum over all vertex weights is needed.

\begin{lemma}[Rank-one projection widths]\label{lem:width2}
For a correction $P=\sum_{j=1}^q\kappa_jc_jc_j^\top\preceq D$, let $w_j$ be the range of $c_j^\top x$ over $B(M)$. Then
\begin{equation}\label{eq:width2}
\kappa_jw_j^2\le2\Lambda\le4L_0.
\end{equation}
\end{lemma}
\begin{proof}
Choose bases $B^+,B^-$ attaining the maximum and minimum projection, and write $v=\one_{B^+}-\one_{B^-}$. Since $\kappa_jc_jc_j^\top\preceq P\preceq D$,
\[
\kappa_jw_j^2\le v^\top Dv
=\sum_{i\in B^+\triangle B^-}b_i
\le b(B^+)+b(B^-)\le2\Lambda.
\]
Apply Lemma~\ref{lem:scale-improved}.
\end{proof}

The same branch optimizer can now use $L_0$ in its tolerance. Only the grid scale changes; the coverage conversion and the exact output certificate remain valid.

\begin{proposition}[Fixed-parameter supplied-certificate algorithm]\label{prop:fpt}
Under the assumptions of Theorem~\ref{thm:master}, for any rational $0<\varepsilon<3/4$ the same certified branch algorithm attains its stated factor and outputs a valid upper bound on $\OPT$ in time
\begin{equation}\label{eq:time-fpt}
(2+q/\varepsilon)^{O(q+1)}N^{O(1)},
\end{equation}
where $N$ is the total rational input length, including the certificate, and the exponent on $N$ is independent of $q$.
\end{proposition}
\begin{proof}
Set the smoothing/optimization tolerance to $\zeta=\varepsilon L_0/4$. In each nonconstant correction coordinate divide the projection interval into
\[
N_j=\max\left\{1,\left\lceil
\sqrt{\frac{q\kappa_jw_j^2}{8\zeta}}\right\rceil\right\}
\]
subintervals; retain the endpoints. Every basis projection has total penalty at most $\zeta$ at one grid point, so the maximum over branch optima bounds $\OPT$, exactly as in Lemma~\ref{lem:grid}. By Lemma~\ref{lem:width2}, the number of grid endpoints is at most
\begin{equation}\label{eq:count-new}
\prod_{j=1}^q(N_j+1)\le
\left(2+\sqrt{\frac{2q}{\varepsilon}}\right)^q.
\end{equation}

The smoothing, conditional-gradient gap, and dyadic-quantization implementation of Section~\ref{sec:algorithm} accepts any such positive rational tolerance. Its analytic bounds use $\sum_i b_i/\zeta\le 8n/\varepsilon$ by Lemma~\ref{lem:scale-improved}; the per-branch bit, oracle, and iteration bounds are consequently polynomial in $N,n,q,1/\varepsilon$, with exponent independent of $q$. 

Choose the branch with the largest certified upper value $\overline U$ and round. The additive loss is at most $\frac52\zeta\le\varepsilon\overline U$, since $L_0\le\OPT\le\overline U$. This proves both the factor and the parameterized running time.
\end{proof}

\section{A stronger spectral coverage conversion}\label{sec:spectral}
The preceding conversion treats the nonlinear part uniformly. We now bound it through the normalized adjacency spectrum, which can be smaller than the maximum vertexwise ratio called total curvature. This produces the final supplied-certificate factor.

After deleting matroid loops as in Proposition~\ref{prop:loops}, delete also vertices with $b_i=0$. They carry no positive-weight edge or modular score, so the optimum of the restricted matroid is unchanged; a basis returned in the restricted matroid extends to an original basis at no change in value. In the rest of this section, $D$ is positive definite. Put
\begin{equation}\label{eq:rho}
\varrho=\lambda_{\max}(D^{-1/2}AD^{-1/2}),
\qquad c=\max_i\frac{d_i}{b_i}.
\end{equation}
The symbol $\varrho$ is different from the edge-addition parameter $\rho$ in Section~\ref{sec:robust}.

The identity below isolates exactly what must be nonnegative for the improved conversion. It replaces the scalar diagonal bound by a spectral bound on the adjacency matrix. We write $z_+=\max\{z,0\}$ for the scalar positive part.

\begin{lemma}[Spectral curvature and exact residual]\label{lem:spectral-residual}
For a nonzero coverage-edge objective, $0<\varrho\le c\le1$ and
$-\varrho D\preceq A\preceq\varrho D$.
For every $x\in[0,1]^V$ and every $u\in[0,1]$,
\begin{equation}\label{eq:residual}
F(x)-(1-u/2)L(x)-uT(x)
=\frac12x^\top(uD-A)x
 +(1-u/2)\sum_{ij\in E}w_{ij}(x_i+x_j-1)_+.
\end{equation}
Consequently, for rational $u\in(0,1]$ with $uD-A\succeq0$,
\begin{equation}\label{eq:spectral-scalar}
uT(x)\le F(x)-(1-u/2)L(x).
\end{equation}
\end{lemma}
\begin{proof}
For every real $z$,
\[
z^\top(cD-A)z=\sum_{ij}w_{ij}(z_i-z_j)^2+
\sum_i(cb_i-d_i)z_i^2\ge0.
\]
Thus $\varrho\le c$. Here $|z|$ denotes coordinatewise absolute value. Because the adjacency matrix is nonnegative entrywise,
$|z^\top Az|\le |z|^\top A|z|\le\varrho z^\top Dz$. This proves the two-sided PSD bound. A positive edge makes $\varrho>0$.

Now $L(x)=b^\top x-\sum_{ij}w_{ij}(x_i+x_j-1)_+$ and $T(x)=\frac12(b^\top x-x^\top Dx)$. Substitute these identities into $F-(1-u/2)L-uT$ and cancel the linear terms to obtain the displayed equality. When $uD-A\succeq0$, both terms on its right are nonnegative.
\end{proof}

\subsection{Concavity only along feasible directions}\label{sec:affine}
The full-space condition is stronger than the optimizer needs. If an affine equation is fixed on every basis, the algorithm never moves in a direction that violates it. The following proposition restricts the curvature test accordingly.
\begin{proposition}[Only feasible directions need nonpositive curvature]\label{prop:affine}
Suppose rational affine equations $Hx=h$ hold for all bases. In Theorem~\ref{thm:master}, the condition $A+\tau D+P\succeq0$ can be replaced by
\[
 z^\top(A+\tau D+P)z\ge0\qquad\text{for every }z\in\ker H.
\]
The other assumptions, including $0\preceq P\preceq D$, remain unchanged.
\end{proposition}
\begin{proof}
All segments in the base polytope have directions in $\ker H$. Thus $G_a$ is concave on this polytope, which suffices for the cap, smoothing, line search, and gap bound. The remaining estimates use only the unchanged assumptions. Let $N$ denote a rational nullspace-basis matrix in this proof. It gives an exact test $N^\top(A+\tau D+P)N\succeq0$. For each row of $H$, equality on all bases can be checked with a minimum and maximum greedy optimization.
\end{proof}
This supplied-equation variant does not itself choose the affine space. Section~\ref{sec:intrinsic} computes the span of all basis differences and determines the correction rank on that space.

\subsection{The stronger supplied-certificate theorem}
We combine the spectral inequality, the feasible-value grid, and concavity on the actual feasible directions. The parameter $u$ bounds the objective spectrum; the parameter $t$ controls the correction. They have different roles.

\begin{theorem}[Spectral supplied-certificate factor]\label{thm:joint}
Suppose the rational correction from Theorem~\ref{thm:master} satisfies the weaker condition
\begin{equation}\label{eq:restricted-certificate}
 v^\top(A+tD+P)v\ge0
 \quad \text{for all }v\in\mathcal T,
 \qquad \mathcal T=\operatorname{span}\{\one_B-\one_{B'}:B,B'\in\mathcal B\},
\end{equation}
with $0\preceq P\preceq D$ and rational $t\in[0,1]$. Let $u\in(0,1]$ be rational and satisfy $uD-A\succeq0$. For every rational $\varepsilon\in(0,3/4)$ there is a deterministic certified algorithm returning a basis $S$ and $\overline U\ge\OPT$ such that
\begin{equation}\label{eq:joint}
 f(S)\ge\left(1-\frac{ut}{2(u+t)}-\varepsilon\right)\overline U,
\end{equation}
in time of the form~\eqref{eq:time-fpt}.
\end{theorem}
\begin{proof}
Use the grid, the concave functions $\Phi_a$, and the certified branch optimizer from Sections~\ref{sec:algorithm} and~\ref{sec:fpt}, with $t$ in place of $\tau$. Condition~\eqref{eq:restricted-certificate} suffices: every segment of $B(M)$ is parallel to $\mathcal T$, so $G_a$ and its capped, smoothed form are concave on the feasible region. The objective and derivative bounds from Section~\ref{sec:algorithm} continue to use only $0\preceq P\preceq D$. 

For $h=\Phi_a(x)$, drop the nonnegative projection penalty, apply Lemma~\ref{lem:spectral-residual}, and use $h\le L(x)$:
\[
\begin{aligned}
h&\le F(x)+tT(x)+\zeta\\
&\le(1+t/u)F(x)-(t/u)(1-u/2)L(x)+\zeta\\
&\le(1+t/u)F(x)-(t/u)(1-u/2)h+\zeta.
\end{aligned}
\]
Rearrangement gives the pointwise inequality
\begin{equation}\label{eq:point-joint}
 F(x)\ge\left(1-\frac{ut}{2(u+t)}\right)\Phi_a(x)
 -\frac{u}{u+t}\zeta.
\end{equation}

As in Proposition~\ref{prop:fpt}, the largest certified branch value upper-bounds $\OPT$; the selected branch is within $3\zeta/2$ of that value. Objective-preserving basis-exchange rounding completes the construction. Because $0\le 1-ut/[2(u+t)]\le1$ and $u/(u+t)\le1$, the total additive loss is at most $5\zeta/2\le\varepsilon\overline U$. The same grid and per-branch estimates prove~\eqref{eq:time-fpt}.
\end{proof}

\section{Intrinsic matroid directions and correction rank}\label{sec:intrinsic}
Theorem~\ref{thm:joint} is useful once a correction is supplied. We now compute the feasible-direction space, identify the necessary rank, and construct a rational correction of no greater rank after a small threshold increase. The computational step is proved in Appendix~\ref{app:rational-correction}; it converts the preceding supplied-certificate result into an input-based theorem.

\subsection{Computing the directions that the algorithm can use}
The exchange graph of one basis identifies which coordinate sums are fixed over every basis. Its connected components therefore give the entire direction space, not only a selected family of feasible exchanges.

\begin{lemma}[The feasible-direction space is oracle computable]\label{lem:tangent}
Fix a basis $B$. Make a graph on $V$ with edges $ij$ for $i\in B$, $j\notin B$ whenever $B-i+j$ is a basis. Let $C_1,\ldots,C_h$ be its connected components. Then
\begin{equation}\label{eq:tangent}
 \mathcal T=\{v\in\R^V:v(C_a)=0\text{ for }a=1,\ldots,h\}.
\end{equation}
One can construct a rational full-column-rank matrix $Z$ spanning $\mathcal T$ with at most $r(n-r)$ basis-exchange independence tests, after computing the initial basis.
\end{lemma}
\begin{proof}
Write $r(S)$ for the matroid rank of a set $S$. For each nonbasis element $j$, the fundamental circuit, the unique minimal dependent subset of $B+j$, consists of $j$ and the basis elements $i$ for which $B-i+j$ is a basis. Every component $C_a$ is therefore spanned by $B\cap C_a$ in the matroid closure, so its rank is $|B\cap C_a|$. The component ranks sum to $r(V)$. Every other basis $B'$ has $|B'\cap C_a|\le r(C_a)$ and the sums of both sides equal $r(V)$, so equality holds in every component. Thus every basis difference has zero sum in each component.

Conversely, every edge $ij$ of the exchange graph yields a feasible basis difference $e_i-e_j$ (up to sign). Edge differences along a connected graph span its coordinate-sum-zero vectors. Taking all components proves the reverse inclusion. Choose one root $i_a\in C_a$ and columns $e_i-e_{i_a}$ for the other vertices of $C_a$; these columns give $Z$.
\end{proof}

Put $s=\dim\mathcal T$ and assume $s>0$. Let
\begin{equation}\label{eq:pencil}
 K=Z^\top DZ\succ0,
 \qquad H_M=K^{-1/2}Z^\top AZK^{-1/2}.
\end{equation}
The eigenvalues $\mu_1\le\cdots\le\mu_s$ of $H_M$ are independent of the rational choice of basis $Z$. They lie in $[-\varrho,\varrho]$ by Lemma~\ref{lem:spectral-residual}. For a symmetric matrix $Q$, write $n_-(Q)$ for its number of strictly negative eigenvalues.

The normalized restricted matrix separates the rank budget from the threshold. A correction must cover every strictly negative direction, and correcting the negative spectral part is sufficient.

\begin{proposition}[Exactly how many directions need correction]\label{prop:inertia}
For rational $t\ge0$, the minimum rank of a \emph{real} matrix $P$ satisfying
\[
0\preceq P\preceq D,
\qquad v^\top(A+tD+P)v\ge0\quad(v\in\mathcal T)
\]
is
\begin{equation}\label{eq:inertia}
 q_M(t)=n_-(Z^\top(A+tD)Z)
 =\#\{i:\mu_i<-t\}.
\end{equation}
Consequently, for a rank budget $q\ge0$, the minimum real threshold is
\begin{equation}\label{eq:tauq}
\tau_q=\begin{cases}\max\{0,-\mu_{q+1}\},&q<s,\\0,&q\ge s.\end{cases}
\qquad 0\le\tau_q\le\varrho.
\end{equation}
\end{proposition}
\begin{proof}
The negative eigenspace of the symmetric form $Z^\top(A+tD)Z$ has dimension $k=q_M(t)$. A correction of rank smaller than $k$ has a nonzero kernel vector on this space, where the form remains negative; hence its rank must be at least $k$.

For sufficiency, set $U=ZK^{-1/2}$, so $U^\top DU=I$, and put $R=(-H_M-tI)_+$, where the positive spectral part replaces negative eigenvalues by zero. The matrix
\[
P_*=DU R U^\top D
\]
is positive semidefinite with rank $k$. Since the eigenvalues of $R$ are in $[0,\varrho]\subseteq[0,1]$, the matrix $D^{-1/2}P_*D^{-1/2}$ is bounded by the identity: the columns of $D^{1/2}U$ are orthonormal. Moreover,
$U^\top(A+tD+P_*)U=H_M+tI+R\succeq0$.
This proves equality in~\eqref{eq:inertia};~\eqref{eq:tauq} follows directly by ordering the eigenvalues.
\end{proof}

\subsection{Rational certificate construction with threshold slack}
The minimum-rank correction is initially real. The next lemma allows a small increase in threshold to create strict margins. It constructs a rational factor with the prescribed number of columns; approximating that factor preserves the rank bound and both matrix inequalities. Exact univariate algebra supplies the needed approximation algorithm, not only existence of a short representation.

\begin{lemma}[Polynomial-time rational low-rank correction with slack]\label{lem:rational-correction}
Let $t_0\in[0,1]$ be rational, let $k=q_M(t_0)$, and let $\delta\in(0,1]$ be rational. A deterministic algorithm computes rational vectors $a_1,\ldots,a_k$ and nonnegative rational factor weights (zero factors may be omitted) defining a rank-at-most-$k$ matrix $\widehat P$, and a rational $\widehat t$ satisfying
\begin{equation}\label{eq:rational-certificate}
0\preceq\widehat P\preceq D,
\quad
Z^\top(A+\widehat tD+\widehat P)Z\succeq0,
\quad
t_0\le\widehat t\le\min\{1,t_0+\delta\}.
\end{equation}
The running time and the rational-factor encoding length are polynomial in the rational input encoding and $\log(1/\delta)$, with absolute polynomial exponents.
\end{lemma}
\begin{proof}
Appendix~\ref{app:rational-correction} gives the complete construction. Lemma~\ref{lem:factor-computation} computes a rational approximation to a factor of a shrunken real correction, with exactly the permitted column count. The positive margins then preserve both matrix inequalities. The proof bounds all intermediate encodings and the computation time; checking a proposed factor alone would not establish this algorithmic statement.
\end{proof}

The factor construction in Appendix~\ref{app:rational-correction} uses polynomial factorization over $\Q$~\cite{LLL82} and univariate real-root algorithms~\cite[Chapter~10]{BPR06}. Its eigenspace computations stay in one single-root number field at a time, with unnormalized bases. Sobczyk's eigenvalue, gap, and projector algorithms~\cite[Corollaries~4.2--4.5 in the full version]{Sobczyk25} are relevant alternatives, but are not invoked as a black-box theorem for the required exact-column-count factor. The present appendix proves that reduction directly.

\subsection{An input-based approximation--complexity tradeoff}
The rank threshold can be approximated by exact inertia bisection. Lemma~\ref{lem:rational-correction} then produces a suitable correction in polynomial bit complexity. Combining these steps with the supplied-certificate theorem gives the following guarantee, with all threshold and optimization errors included in $\varepsilon$.

\begin{theorem}[Intrinsic spectral tradeoff]\label{thm:intrinsic-main}
Consider weighted graph coverage plus nonnegative modular rewards under an arbitrary matroid, with rational input and an independence oracle. Let $\varrho$ and $\tau_q$ be defined by~\eqref{eq:rho} and~\eqref{eq:tauq} after objective-preserving preprocessing. For an integer rank budget $q\ge0$ and rational $0<\varepsilon<3/4$, a deterministic algorithm returns an independent set $S$ and a rational upper bound $\overline U$ such that
\begin{equation}\label{eq:intrinsic-result}
\boxed{\quad \OPT\le\overline U,
\qquad f(S)\ge\left(1-\frac{\varrho\tau_q}{2(\varrho+\tau_q)}-\varepsilon\right)\overline U.\quad}
\end{equation}
The loss fraction is interpreted as zero when $\varrho=\tau_q=0$. The runtime and independence-oracle complexity have the form
\begin{equation}\label{eq:intrinsic-runtime}
(2+q/\varepsilon)^{O(q+1)}N^{O(1)},
\end{equation}
where $N$ is the binary input length. The exponent of $N$ is independent of $q$. The automatic construction uses Lemma~\ref{lem:rational-correction}, whose exact-algebraic algorithm has polynomial bit complexity independently of $q$.
\end{theorem}
\begin{proof}
Use Proposition~\ref{prop:loops} to delete loops while preserving their incident rewards; delete remaining zero-singleton elements and later extend a solution to the original matroid. If all rewards vanish, solve exactly. If $A=0$, weighted matroid greedy is exact. If $s=0$, the reduced matroid has a unique basis.

Otherwise compute $Z$ by Lemma~\ref{lem:tangent}. Let $\delta=\varepsilon/8$. Use exact rational PSD bisection to find $u\in[\varrho,\varrho+\delta]$ with $uD-A\succeq0$; one may cap $u$ at $c$. Similarly use exact rational inertia bisection on $Z^\top(A+tD)Z$ to find $t_0\in[\tau_q,\tau_q+\delta]$, capped at one, with $q_M(t_0)\le q$. Both bisections have polynomial rational bit complexity. Apply Lemma~\ref{lem:rational-correction} at $(t_0,\delta)$. It returns a rank-at-most-$q$ rational correction at $\widehat t\in[t_0,\min(1,t_0+\delta)]$, including the zero-correction boundary cases. Its factor construction and exact PSD acceptance are proved in Appendix~\ref{app:rational-correction}. Apply Theorem~\ref{thm:joint} with internal accuracy $\varepsilon/2$.

For $u,t\ge0$ not both zero, write $\alpha(u,t)=1-ut/[2(u+t)]$. Its partial derivatives are bounded in absolute value by $1/2$:
\[
\left|\frac{\partial\alpha}{\partial u}\right|
=\frac{t^2}{2(u+t)^2}\le\frac12,
\qquad
\left|\frac{\partial\alpha}{\partial t}\right|
=\frac{u^2}{2(u+t)^2}\le\frac12.
\]

The continuous extension at $(0,0)$ obeys the same coordinatewise bound. Here $u-\varrho\le\delta$ and $\widehat t-\tau_q\le2\delta$, so
$\alpha(u,\widehat t)\ge\alpha(\varrho,\tau_q)-3\delta/2$. Including the internal loss $\varepsilon/2$, the total loss is at most $11\varepsilon/16<\varepsilon$. Objective-preserving basis extension does not alter the value or upper certificate. Lemma~\ref{lem:rational-correction}, the exact PSD tests, and Proposition~\ref{prop:fpt} give the runtime~\eqref{eq:intrinsic-runtime}.
\end{proof}

\section{Robustness to edges that destroy bipartiteness}\label{sec:robust}
The supplied-certificate formulation also identifies graph changes that preserve a useful correction. We keep the matroid unrestricted and control added edge weight locally at each vertex.

\begin{theorem}[Local edge-addition stability]\label{thm:robust}
Suppose $(A_0,D_0,P,\tau_0)$ satisfies \eqref{eq:cert}. Add nonnegative weighted edges with adjacency matrix $E$ and degree matrix $D_H$. If
\[
 D_H\preceq\rho D_0\qquad(\rho\ge0),
\]
then the new instance has the same correction $P$ and a valid parameter
\begin{equation}\label{eq:tauprime}
 \tau'=\frac{\tau_0+\rho}{1+\rho}.
\end{equation}
The correction dimension does not increase.
\end{theorem}
\begin{proof}
For all real $z$,
\[
 z^\top(E+D_H)z=\sum_{ij\text{ added}}w_{ij}(z_i+z_j)^2\ge0.
\]
Writing $A'=A_0+E$ and $D'=D_0+D_H$, expand
\begin{align*}
 A'+\tau'D'+P
 &=A_0+\tau_0D_0+P+(E+D_H)\\
 &\quad +(\tau'-\tau_0)D_0-(1-\tau')D_H.
\end{align*}
The last line equals
$\frac{1-\tau_0}{1+\rho}(\rho D_0-D_H)\succeq0$.
All three terms are positive semidefinite, and $P\preceq D_0\preceq D'$. This proves the certificate.
\end{proof}
The assumption controls added weight at each vertex. A bound on total added weight alone does not imply it. Added edges can join vertices on the same side of an original bipartition.

\subsection{Bipartite spectral cores and their perturbations}
A connected bipartite core gives an explicit rank-one correction. The stability theorem then extends that certificate to graphs that need not remain bipartite.

For a connected bipartite graph with positive weighted degrees, let $s$ take value $+1$ on one side and $-1$ on the other, $W$ its total edge weight, and $D$ its degree matrix. If the second eigenvalue of $I-D^{-1/2}AD^{-1/2}$ is at least $\gamma\in(0,1]$, define
\begin{equation}\label{eq:bipcert}
 \tau_0=1-\gamma,\qquad
 P=\frac{\gamma}{2W}(Ds)(Ds)^\top.
\end{equation}

This is a rank-one certificate. Indeed, the normalized adjacency matrix has eigenvector $D^{1/2}s/\sqrt{2W}$ at eigenvalue $-1$. Bipartite symmetry and the gap promise put all other eigenvalues, apart from $1$, in $[-(1-\gamma),1-\gamma]$. The rank-one term raises the $-1$ eigenvalue by $\gamma$, while the diagonal term shifts all eigenvalues by $1-\gamma$. Hence \eqref{eq:cert} holds. Also $D^{-1/2}PD^{-1/2}$ is $\gamma$ times an orthogonal rank-one projector, proving $P\preceq D$.

Combining Theorem~\ref{thm:robust} with Theorem~\ref{thm:master} gives the following result for arbitrary matroids, including matroid loops by Proposition~\ref{prop:loops}.
\begin{corollary}[Perturbed bipartite spectral core]\label{cor:robust}
Let a connected bipartite core have normalized gap at least $\gamma\in(0,1]$. Add edges with incident weight at most $\rho$ times each original vertex degree. Then a deterministic polynomial-time algorithm attains
\begin{equation}\label{eq:robustfactor}
 \left(\frac{3(1+\rho)-\gamma}{4(1+\rho)-2\gamma}
 -\varepsilon\right)\OPT
\end{equation}
for the coverage objective on the augmented graph under any matroid.
\end{corollary}
For $\gamma=1/2$ and $\rho=1/10$, the factor is $14/17-\varepsilon$. With $\varepsilon=1/32$, it is $431/544\approx0.792279$. The augmented graph can contain triangles. The value $\gamma/(1+\rho)$ is a certificate parameter; it is not claimed to equal the augmented graph's normalized Laplacian gap.

\subsection{Disconnected cores}
Independent graph components contribute separate correction directions, even when the matroid couples their selections. This gives a direct interpretation of the correction budget.

For $q$ disjoint connected bipartite components with gaps at least a common $\gamma$, take the sum of their rank-one corrections in \eqref{eq:bipcert}. The certificate is block diagonal, has at most $q$ terms, and uses $\tau=1-\gamma$. The approximation factor is unchanged. The matroid can couple different components; the algorithm optimizes over the full base polytope. Its running time is polynomial when the number of components is fixed, not uniformly polynomial when that number grows.

\section{Consequences, examples, and limits of the parameter}\label{sec:examples-new}
These comparisons distinguish three different claims: a curvature-only specialization, a separation between full and thresholded inertia, and a reduction from range-relative quadratic approximation. None is a hardness theorem or a proof of historical priority.

The diagonal bound alone supplies a zero-correction baseline; this is a specialization rather than a novelty claim.

\begin{corollary}[Curvature-only bound]
For $c=\max_i d_i/b_i$ and no correction terms ($q=0$), $A+cD\succeq0$. Hence a polynomial-time version of Theorem~\ref{thm:joint} gives $(1-c/4-\varepsilon)\OPT$. If $c=0$, the objective is modular and greedy is exact.
\end{corollary}
\begin{proof}
The identity $v^\top(A+cD)v=\sum_{ij}w_{ij}(v_i+v_j)^2+\sum_i(cb_i-d_i)v_i^2\ge0$ supplies the certificate. Take $u=c$ and $t=c$ in~\eqref{eq:joint}; the resulting factor is $1-c/4$. The modular case is immediate.
\end{proof}

Counting every negative eigenvalue can overstate the search dimension needed for a nonzero threshold. Disjoint triangles make the distinction explicit within the matroid direction space.

\begin{proposition}[Full versus thresholded negative inertia]\label{prop:triangle}
For each $m\ge1$, take $m$ disjoint unit-weight triangles and impose the rank-$m$ uniform matroid on their $3m$ vertices. Then
\[
\varrho=1,\qquad q_M(0)=2m,\qquad q_M(1/2)=0,
\qquad\alpha(1,1/2)=5/6.
\]
\end{proposition}
\begin{proof}
Let $J_3$ be the $3\times3$ all-ones matrix. The adjacency of one triangle is $J_3-I_3$ with eigenvalues $2,-1,-1$, while its degree matrix is $2I_3$. The tangent space of the uniform matroid consists of all vectors with total coordinate sum zero. Each triangle contributes two independent negative eigenvectors, all supported on its vertices and with coordinate sum zero; the remaining block-constant subspace has no negative directions. Thus the restricted form has $2m$ negative eigenvalues at threshold zero. 

At $t=1/2$, the corrected matrix is $A+I$ with eigenvalues $3,0,0$ per block, so no negative directions remain. Normalization by $D=2I$ gives $\varrho=1$.
\end{proof}
This proves separation between the \emph{parameters} $q_M(0)$ and $q_M(t)$, not a lower bound against other algorithms or a new approximation hardness threshold.

To compare with range-relative quadratic optimization, we must relate the continuous objective range to $\OPT$. Objective-preserving rounding gives exactly the inequality needed for that conversion.

\begin{proposition}[A comparison with range-relative indefinite programming]\label{prop:range}
On $B(M)$, $0\le F(x)\le\OPT$ and $\max_{x\in B(M)}F(x)=\OPT$. Thus a feasible fractional algorithm achieving
\[
F(x)\ge\max_{B(M)}F-\delta\bigl(\max_{B(M)}F-\min_{B(M)}F\bigr)
\]
can be followed by matroid exchange rounding to obtain a $(1-\delta)$ approximation.
\end{proposition}
\begin{proof}
Each modular and edge term of $F$ is nonnegative on $[0,1]^V$. Section~\ref{sec:algorithm} proves that every $x\in B(M)$ rounds to a basis with value at least $F(x)$, so $F(x)\le\OPT$; an optimal basis indicator attains equality. The range of $F$ on $B(M)$ is at most $\OPT$. Apply the assumed fractional inequality and round.
\end{proof}
This explains the close connection to Vavasis's range-relative approximation framework~\cite{Vavasis92}. Our candidate contribution is the supplied-certificate coverage guarantee and its thresholded, feasible-direction extension~\eqref{eq:intrinsic-result}, not the general principle of approximating indefinite quadratic forms.

\section{Evidence, computational dependencies, and scope}\label{sec:verification}
The implementation has two distinct certification levels. We state what its checks establish before discussing the broader algorithmic claim and the remaining priority comparison.

The rational optimization and exchange-rounding engine checks objective inequalities and produces exact rational certificates for supplied corrections. The older automatic-certificate prototype proposes factors with numerical eigendecomposition and accepts candidates only after exact rational rank and PSD checks. It has a finite precision guard and may report failure. No unconditional worst-case runtime is claimed for that numerical prototype.

The automatic algorithm proved in Appendix~\ref{app:rational-correction} is instead exact-algebraic. A separate small-instance reference implementation isolates real algebraic roots, keeps each eigenspace calculation in a single-root field, encloses scalar square roots rationally, and checks the output inequalities exactly. Fourteen regression cases cover repeated and irrational eigenvalues, nonorthogonal tangent coordinates, a negative eigenvalue $2^{-80}$ below a threshold, highly unequal diagonal scales, and the zero-rank and capped-threshold cases. These checks exercise the new construction; the asymptotic theorem uses the polynomial-time algebraic algorithms cited in the proof, not a claim about every internal routine of the reference symbolic library.

The previous verification suites include exact arithmetic checks of the scalar residual, grid coverage, curvature conversion, certificate preservation, oracle-derived tangent spaces, and small exhaustively checked matroid instances. They are regression tests, not a formal proof assistant or independent review. The archived scripts, input parameters, random seeds, and JSON result files specify the reproduction conditions. The paper's theorems do not rely on the outcomes of those finite tests.

\paragraph{Distinction from existing results.}
Sellier~\cite{Sellier23} gives a fixed-parameter approximation scheme for frequency-bounded coverage parameterized by \emph{matroid rank}; this is different from $q_M(t)$. Sviridenko, Vondr\'ak, and Ward~\cite{SVW17} develop curvature-sensitive submodular approximation. Together with the spectral and convex-domain predecessors discussed in the introduction, these results delimit the candidate contribution: the supplied-certificate multiplicative formula, its independence-oracle implementation, and the threshold/accuracy extension. Whether that combination follows from prior theorems still requires an independent priority assessment.

\paragraph{Scope of the theorem.}
The correction dimension $q$ may grow with $n$, and the running time may be exponential in $q$. The supplied-certificate theorems and the input-based theorem use exact rational data. The automatic construction requires positive threshold slack, which is charged to the approximation error. It does not claim an exact minimum-rank rational correction at a zero-slack algebraic threshold. The approximation is not a universal improvement for arbitrary monotone submodular maximization. Mathematical review and historical priority assessment remain separate from the finite verification checks; no referee endorsement or unproved hardness claim is asserted.

\noindent\textbf{Declaration of Generative AI Use.} ChatGPT (OpenAI) was used to assist with mathematical analysis and manuscript preparation. The author takes responsibility for the final content.

\label{rev:main-end}
\appendix
\section{Rational branch optimization and output certificates}\label{sec:algorithm}
This appendix supports the supplied-certificate theorem and its stronger versions in the main text. It uses the branch functions $G_a$ and $\Phi_a$, the positive tolerance $\zeta$, and the base polytope $B(M)$ defined there. The purpose is to obtain both a feasible point and a certified branch upper bound with controlled bit lengths. The algorithmic primitives are established: linear optimization by matroid greedy, conditional-gradient gap certificates~\cite{Jaggi13}, and symmetric-exchange rounding~\cite{CVZ09}. We give the error and encoding details needed for this particular objective.

\subsection{A smooth concave lower approximation}
The branch cap is nonsmooth. A rational smoothing lets the algorithm use gradients while keeping an explicit bound on the difference from the original cap.

For rational $\delta>0$, define
\[
 H_\delta(t)=\begin{cases}|t|,&|t|\ge\delta,\\
 t^2/(2\delta)+\delta/2,&|t|\le\delta,
 \end{cases}
 \qquad m_\delta(u,v)=\frac{u+v-H_\delta(u-v)}2.
\]
The function $m_\delta$ is continuously differentiable, concave, and coordinatewise nondecreasing. Its two partial derivatives sum to one. Also
\[
 0\le\min(u,v)-m_\delta(u,v)\le\delta/4.
\]

Let $\mu=\zeta/S_b$ and set
\[
 \widetilde L(x)=\ell^\top x+
 \sum_{ij\in E}w_{ij}m_\mu(1,x_i+x_j),\qquad
 \psi_a(x)=m_\zeta(\widetilde L(x),G_a(x)+\zeta).
\]

Composition preserves concavity here because $m_\delta$ is concave and nondecreasing in each argument. All values and gradients at rational inputs are rational. Since $\sum w_{ij}\le S_b/2$,
\begin{equation}\label{eq:smoothloss}
 0\le\Phi_a(x)-\psi_a(x)\le\zeta/8+\zeta/4<\zeta/2.
\end{equation}

The stopping and encoding analysis needs derivative and range bounds uniform over the grid. Within this appendix, $\Lambda$ denotes the Lipschitz constant below, not the maximum singleton-weight basis value used in Section~\ref{sec:fpt}.

\begin{lemma}[Uniform bounds on the base polytope]\label{lem:bounds}
For centers in the grid, put
\[
 M_1=(q+2)S_b,\quad
 \Lambda=\bigl[1+(q+3)^2\bigr]\frac{S_b^2}{\zeta}+3S_b,
 \quad R=(q+3)S_b.
\]
On $B(M)$, $\|\nabla\psi_a\|_1\le M_1$, the gradient is Euclidean Lipschitz with constant at most $\Lambda$, and the range of $\psi_a$ has length at most $R$.
\end{lemma}
\begin{proof}
For $x\in B(M)$, $|c_j^\top x-a_j|\le w_j$. Apply $\kappa_jc_jc_j^\top\preceq D$ to the sign vector of $c_j$ to obtain $\kappa_j\|c_j\|_1^2\le S_b$. Lemma~\ref{lem:grid} gives
$\kappa_jw_j\|c_j\|_1\le S_b$.

The gradient of $F$ has nonnegative coordinates bounded by $b_i$. The gradient of $T$ has absolute coordinates at most $b_i/2$. Thus $\|\nabla G_a\|_1\le(q+3/2)S_b$. Likewise $\|\nabla\widetilde L\|_1\le S_b$. The outer smoothing takes a convex combination of these gradients, proving the stated $M_1$ bound.

The operator norms of $A$, $\tau D$, and $P$ are each at most $S_b$, so $\|\nabla^2G_a\|_{\rm op}\le3S_b$. An edge smoothing term contributes Hessian norm at most $w_{ij}/\mu$; summing gives at most $S_b^2/\zeta$. The outer smoothing adds a negative rank-one Hessian whose norm is at most $\|\nabla\widetilde L-\nabla G_a\|_2^2/(2\zeta)$. The preceding gradient bounds imply the displayed $\Lambda$. These bounds also apply across the smoothing breakpoints because the gradient is continuous and piecewise Lipschitz.

Finally $0\le F,L\le S_b$, $0\le T\le S_b/8$, and the penalty is at most $qS_b/2$. The cap and its smoothing therefore have range length at most $(q+3)S_b$, using $\zeta<3S_b/16$.
\end{proof}

\subsection{Certified branch optimization}
A greedy linear optimization bounds the error of the current iterate. We first derive that bound, then specify updates that reach it in the rational input model.

Maintain $x$ as a convex combination of bases. At a current point choose a basis indicator $v$ maximizing $\nabla\psi_a(x)^\top v$ by greedy. Let
\[
 g=\nabla\psi_a(x)^\top(v-x)\ge0.
\]
Concavity and linear optimization imply
\begin{equation}\label{eq:gap}
 \max_{z\in B(M)}\psi_a(z)\le\psi_a(x)+g.
\end{equation}

Stop the branch when $g\le\zeta$. Define its certified upper bound
\begin{equation}\label{eq:ua}
 u_a=\psi_a(x)+g+\zeta/2.
\end{equation}
Equations~\eqref{eq:smoothloss}--\eqref{eq:gap} show $u_a\ge\max\Phi_a$ and
\begin{equation}\label{eq:recoverbranch}
 \Phi_a(x)\ge u_a-3\zeta/2.
\end{equation}

\paragraph{Termination and bit growth.}
The following explicit implementation avoids an unstated real-arithmetic model. Set $C=n\Lambda$. The squared diameter of $B(M)$ is at most $n$. Choose a fixed dyadic step mesh with spacing at most $\zeta/(2C)$. Concavity permits derivative bisection on a line to find a maximum over that mesh. If $g>\zeta$, smoothness gives
\[
 \psi_a(x+t(v-x))\ge\psi_a(x)+tg-Ct^2/2.
\]
Here $g\le M_1\le C$. A mesh point in $[g/(2C),g/C]$ exists, so the best mesh point gains at least $g^2/(4C)>\zeta^2/(4C)$.

Choose one basis in the maintained mixture as the anchor. After a step, quantize all nonanchor mixture weights downward to a fixed dyadic denominator $Q$ and assign the remainder to an anchor basis. If the mixture has at most $s$ terms, the change of each coordinate is at most $s/Q$. Thus the loss is at most $M_1s/Q$. The implementation uses
\[
 N_{\rm it}=\left\lceil\frac{32RC}{\zeta^2}\right\rceil+1,\quad
 s_{\max}=1+|\mathcal A|(N_{\rm it}+1),\quad
 Q\ge\frac{16M_1Cs_{\max}}{\zeta^2},
\]
with $Q=2^k$ for an integer $k\ge0$. Each nonterminating update still gains more than $\zeta^2/(8C)$. The range bound forces termination before $N_{\rm it}$ updates per branch. Warm starts can be retained between branches; the chosen $s_{\max}$ bounds their accumulated supports.

The mesh and $Q$ have polynomial bit length for fixed $q$. A tentative step multiplies only two fixed dyadic denominators; quantization then returns to $Q$. Rational Huber evaluations, gradients, gaps, and comparisons have polynomial encoding length. The ratio $S_b/B_0\le n$ makes the iteration bound polynomial in $n$, $q$, and $1/\varepsilon$, apart from the explicit branch dependence. This establishes the stated bit and oracle model, rather than only an arithmetic-operation bound.

\subsection{Deterministic exchange rounding}
The optimizer retains its own convex decomposition, so no additional decomposition oracle is needed. We merge its bases while comparing the explicit objective at the two feasible exchange endpoints.

Given a convex combination of bases, combine two terms with positive coefficients $\lambda,\nu$. If their bases $B_1,B_2$ differ, symmetric basis exchange supplies $i\in B_1\setminus B_2$ and $j\in B_2\setminus B_1$ for which both $B_1-i+j$ and $B_2-j+i$ are bases. These can be found by testing at most $r$ candidates with two independence queries each. The two possible replacements change the aggregate point to
\[
 x^-=x-\lambda(e_i-e_j),\qquad
 x^+=x+\nu(e_i-e_j).
\]

The current $x$ is a convex combination of these points. Along this line, the quadratic coefficient of $F$ is $w_{ij}\ge0$ (zero if there is no such edge). Therefore at least one endpoint has value at least $F(x)$. Choose the better endpoint. 

The two bases acquire one more common element. After at most $r$ exchanges they coincide and their coefficients can be combined. Repeating yields a basis $S$ with $f(S)\ge F(x)$ in polynomial time in the explicit support size. The modular scores cause no additional curvature.

\begin{proof}[Proof of Theorem~\ref{thm:master}]
Run the certified branch method and round its terminal point for every grid center. Return the best basis and put $\overline U=\max_a u_a$. Equation~\eqref{eq:global-upper} proves $\overline U\ge\OPT$. Consider the branch attaining $\overline U$. By rounding, Lemma~\ref{lem:conversion}, and \eqref{eq:recoverbranch}, the returned value is at least
\[
 \alpha\overline U-
 \left(\frac{3\alpha}{2}+\frac1{1+\tau}\right)\zeta
 \ge\alpha\overline U-\frac52\zeta
 \ge(\alpha-\varepsilon)\overline U.
\]
The final inequality uses $\zeta=\varepsilon B_0/4$ and $B_0\le\OPT\le\overline U$. Lemma~\ref{lem:grid} and the rational optimization analysis prove the complexity claim.
\end{proof}

\section{Constructing rational corrections with strict margins}\label{app:rational-correction}
This appendix proves the algorithmic claim in Lemma~\ref{lem:rational-correction}. A short rational factor is not enough: the factor must be found in polynomial bit complexity without adding columns. We first compute the required generalized eigenspaces by exact univariate algebra. We keep their bases unnormalized, and approximate a separate scalar square root only at the last step. The resulting factor then fits the strict margins of the certificate.

Throughout the appendix, $A$ is rational symmetric, $D$ is positive rational diagonal, and $Z$ is rational with linearly independent columns. The inequalities $-D\preceq A\preceq D$ hold by Lemma~\ref{lem:spectral-residual}. Put $s=\rank(Z)$, $Q=Z^\top AZ$, and $K=Z^\top DZ$. If $s>0$, then $K\succ0$. Binary encoding length includes dimensions and the numerators and denominators of all rational entries. Polynomial bounds below have absolute exponents, independent of the number of corrected directions.

\subsection{The exact-algebraic operations needed by the construction}
We use polynomial factorization over $\Q$, not integer factorization. The deterministic polynomial-time factorization theorem of Lenstra, Lenstra, and Lov\'asz~\cite{LLL82}, and univariate real-root algorithms~\cite[Chapter~10]{BPR06}, supply the non-elementary algorithmic prerequisites. The following lemma states the precise encoding convention. It does not assume that arithmetic in a splitting field containing all eigenvalues is efficient.

\begin{lemma}[Single-root arithmetic and approximation]\label{lem:single-root}
Let $p\in\mathbb Z[X]$ be an irreducible polynomial of degree $d$, and specify a real root $\lambda$ by a rational isolating interval. Elements of $\Q(\lambda)$ are represented by polynomials of degree less than $d$ with rational coefficients. Field arithmetic, exact zero and sign tests, and linear algebra with rational-function output admit polynomial bit complexity in the explicitly represented data, whenever the output has polynomial degree and coefficient height.

In particular, determinants and linear-system solutions of polynomial-size matrices with polynomial-height entries have such bounds. If $H$ bounds the coefficient bit lengths of both $p$ and a nonzero field element, that element has absolute value between $2^{-\operatorname{poly}(d,H)}$ and $2^{\operatorname{poly}(d,H)}$. If $a,b,c\in\Q(\lambda)$ have polynomial-height representations, $b(\lambda)>0$, and $c(\lambda)>0$, then
\[
 a(\lambda)\sqrt{b(\lambda)/c(\lambda)}
\]
can be enclosed in a rational interval of width at most $2^{-h}$ in time polynomial in the data encoding and $h$. The constants in these statements do not depend on a separately fixed matrix dimension.
\end{lemma}
\begin{proof}
Reduce polynomial products modulo $p$. Inversion of a nonzero residue uses the extended Euclidean algorithm for that residue and $p$; the residue is invertible because $p$ is irreducible. Subresultant or determinant formulas bound the degrees and coefficient lengths of these operations polynomially. For a matrix of order $m$, clear common denominators before taking its determinant. The determinant has degree at most $m$ times the largest entry degree, and its coefficient length is polynomial in $m$, the entry degree, and the entry length. Cramer's rule gives the same conclusion for solutions. Fraction-free elimination computes the requisite ranks and nonzero minors in polynomial time. These statements concern the displayed matrix problems; they do not assert polynomial height for an arbitrary unbounded-depth arithmetic circuit.

For completeness, a resultant gives the required lower bound on nonzero values. Clear denominators in a field element to write it as $f(\lambda)/a$, with $f\in\mathbb Z[X]$, $\deg f<d$, and $a\ne0$. If $f\ne0$, then $\operatorname{Res}(p,f)$ is a nonzero integer. It equals a power of the leading coefficient of $p$ times the product of $f$ over the $d$ conjugate roots. The root bound for $p$ and the coefficient bound for $f$ give an upper bound $2^{\operatorname{poly}(d,H)}$ for every factor except the selected one. Since the absolute value of the resultant is at least one, the selected value, after division by $a$, is at least $2^{-\operatorname{poly}(d,H)}$. The upper bound follows from the same root bound. Applying this argument to derivatives or the squarefree discriminant also gives a separation bound of that form for distinct roots.

Signed subresultant root counting and interval refinement isolate the real roots and compare them with rational thresholds in polynomial bit complexity~\cite[Chapter~10]{BPR06}. For polynomial-height rational functions, the preceding lower bounds control nonzero denominators, and polynomial derivative bounds control evaluation error. Refining the isolating interval by polynomially many additional bits therefore encloses each field value with prescribed accuracy. A positive radicand has a polynomial-bit positive lower bound, so rational bisection of its scalar square root and interval multiplication enclose the displayed expression with $h$ accurate bits. No new algebraic numbers need to be adjoined to the field during this numerical enclosure step.
\end{proof}

The next argument uses only two bounded-depth linear-algebra stages in a single-root field: a kernel computation and an unnormalized orthogonal projection. This is why the polynomial-height condition in Lemma~\ref{lem:single-root} is available rather than assumed.

\subsection{A rank-preserving factor without a common splitting field}
The target is the real correction $P_*$ in Proposition~\ref{prop:inertia}, evaluated at $t=t_0$. We prove the stronger factor-approximation statement formerly needed as a computational hypothesis. This proof is separate from Sobczyk's spectral projector results~\cite[Corollaries~4.2--4.5 in the full version]{Sobczyk25}; those results are not cited as a black-box proof of factor extraction.

\begin{lemma}[Polynomial-time approximation of the required factor]\label{lem:factor-computation}
Suppose $k=q_M(t_0)>0$, with rational $0\le t_0<1$, and let $\delta,\eta\in(0,1)$ be rational with $t_0+\delta<1$. A deterministic algorithm computes a rational $n\times k$ matrix $\widehat R$ such that, for some real $n\times k$ matrix $R$,
\begin{equation}\label{eq:factor-approximation}
 RR^\top=(1-\delta/4)P_*,
 \qquad \|D^{-1/2}(\widehat R-R)\|_F\le\eta.
\end{equation}
Its bit complexity and output encoding length are polynomial in the rational input encoding, $\log(1/\delta)$, and $\log(1/\eta)$, with absolute polynomial exponents.
\end{lemma}
\begin{proof}
We give the algebraic construction first and its bit bounds afterward. Compute and factor the rational characteristic polynomial
\[
 \chi(X)=\det(Q-XK).
\]
All its roots are real because $K\succ0$ and $Q$ is symmetric. The matrix $K^{-1/2}QK^{-1/2}$ is diagonalizable, so the multiplicity of a root equals the dimension of its generalized eigenspace. Factorization over $\Q$ is deterministic polynomial-time~\cite{LLL82}. Isolate the real roots of each irreducible factor and compare them exactly with $-t_0$. A root equal to $-t_0$ is not selected. The selected multiplicities sum to $k$ by inertia.

Fix one selected root $\lambda$, with minimal polynomial $p$. Work only in $\Q(\lambda)$. Find a basis $w_1,\ldots,w_m$ of
\[
 \ker(Q-\lambda K).
\]
For $j>1$, let $W_{j-1}$ have columns $w_1,\ldots,w_{j-1}$ and define
\begin{equation}\label{eq:unnormalized-projection}
 v_1=w_1,\qquad
 v_j=w_j-W_{j-1}(W_{j-1}^{\top}KW_{j-1})^{-1}
                    W_{j-1}^{\top}Kw_j.
\end{equation}
The prefix Gram matrix is positive definite. Hence each $v_j$ is nonzero, belongs to the same eigenspace, and is $K$-orthogonal to all earlier vectors. Every coordinate of $v_j$ remains in $\Q(\lambda)$. Formula~\eqref{eq:unnormalized-projection} uses the original kernel basis; it does not recursively adjoin square roots of norms.

Write $\nu_j=v_j^\top Kv_j>0$. For this eigenspace, put in $R$ the $m$ columns
\begin{equation}\label{eq:algebraic-factor-column}
 r_{\lambda,j}
 =DZv_j\sqrt{\frac{(1-\delta/4)(-\lambda-t_0)}{\nu_j}}.
\end{equation}
Eigenspaces for distinct roots are $K$-orthogonal: symmetry gives
$(\lambda-\lambda')v^\top Kv'=0$. Thus the vectors $v_j/\sqrt{\nu_j}$, over all roots, form a $K$-orthonormal generalized eigenbasis. Substitution into the spectral formula for $P_*$ in Proposition~\ref{prop:inertia} shows that the $k$ selected columns in~\eqref{eq:algebraic-factor-column} satisfy the first equality in~\eqref{eq:factor-approximation}. This identity determines their column count without any numerical rank test.

We next bound the exact intermediate representations. Clearing rational denominators in $\chi$ gives an integer polynomial of degree $s$ and polynomial coefficient bit length, by the determinant bound in Lemma~\ref{lem:single-root}. Its irreducible factors have polynomial encoding length. For a factor of degree $d\le s$, a kernel basis is obtained from a minor of $Q-XK$ that is invertible modulo $p$, together with its associated column equations. Its coordinates are ratios of determinants of matrices with degree-one polynomial entries. They have polynomial degrees and heights. Reduction modulo $p$ and field inversion preserve polynomial encoding by Lemma~\ref{lem:single-root}.

The prefix Gram matrices in~\eqref{eq:unnormalized-projection} have polynomial-height entries. Their inverses are again ratios of determinants of order at most $s$. Consequently, all $v_j$, $\nu_j$, and $DZv_j$ have polynomial-height single-root representations. There is a fixed number of determinant/field-reduction stages; no tower of fields or common splitting field is constructed. Distinct irreducible factors are processed separately. If $\chi$ has factors $p_\ell^{m_\ell}$, then $\sum_\ell m_\ell\deg p_\ell=s$. The total number of these polynomial-size computations is therefore polynomial in the input dimension.

Each coordinate in~\eqref{eq:algebraic-factor-column} is a polynomial-height field element times one scalar square root of a positive field element. Lemma~\ref{lem:single-root} encloses it with any requested absolute accuracy in polynomial time. In particular, let $b_{\min}=\min_iD_{ii}>0$ and choose
\[
 h=\frac{\eta\min\{1,b_{\min}\}}{2nk}.
\]
Compute a rational interval of width at most $2h$ for each of the $nk$ coordinates, and take its midpoint. The resulting matrix obeys
\[
 \|D^{-1/2}(\widehat R-R)\|_F^2
 \le\frac{nk h^2}{b_{\min}}
 \le\frac{\eta^2}{4nk}\le\eta^2.
\]
The logarithm of $1/h$ is polynomial in the input encoding and $\log(1/\eta)$. The magnitudes of all entries are bounded because $RR^\top\preceq D$. Both their integer parts and their approximating denominators therefore have polynomial length. This proves the error, output-size, and bit-complexity claims. Repeated eigenvalues are handled by their complete unnormalized kernel basis; arbitrarily close distinct eigenvalues require only the additional polynomially bounded isolation precision from Lemma~\ref{lem:single-root}.
\end{proof}

\subsection{Strict margins and exact certificate acceptance}
We can now replace the existence argument by an algorithm. Positive slack is essential: the theorem does not promise an exactly minimum-rank rational correction at a zero-slack algebraic threshold.

\begin{proof}[Proof of Lemma~\ref{lem:rational-correction}]
Compute $k$ by exact rational inertia of $Z^\top(A+t_0D)Z$. If $k=0$, take $\widehat P=0$ and $\widehat t=t_0$. If $t_0+\delta\ge1$, take $\widehat P=0$ and $\widehat t=1$; $A+D\succeq0$ by Lemma~\ref{lem:spectral-residual}. The empty tangent-space case is included in $k=0$.

Assume $k>0$ and $t_0+\delta<1$. Let $P_*\preceq D$ be the real correction from Proposition~\ref{prop:inertia}, and put $P'=(1-\delta/4)P_*$. For every $v\in\mathcal T$,
\begin{equation}\label{eq:margin}
 v^\top(A+(t_0+\delta)D+P')v\ge(3\delta/4)v^\top Dv,
 \qquad 0\preceq P'\preceq(1-\delta/4)D.
\end{equation}
The first inequality adds $\delta D$ to the original feasible form and subtracts only $(\delta/4)P_*\preceq(\delta/4)D$.

Apply Lemma~\ref{lem:factor-computation} with $\eta=\delta/64$. Set $\widehat P=\widehat R\widehat R^\top$ and $\widehat t=t_0+\delta$. This is rational and positive semidefinite with rank at most $k$. Since $RR^\top=P'$ and $\|D^{-1/2}R\|_{\rm op}\le1$, expansion of the product difference gives
\[
 \|D^{-1/2}(\widehat P-P')D^{-1/2}\|_{\rm op}
 \le2\eta+\eta^2<\delta/16.
\]
Combining with~\eqref{eq:margin} yields
\[
 0\preceq\widehat P\preceq(1-3\delta/16)D,
 \qquad
 v^\top(A+\widehat tD+\widehat P)v
 \ge(11\delta/16)v^\top Dv \quad(v\in\mathcal T).
\]
Thus both certificate inequalities hold exactly. The columns of $\widehat R$, each with weight one, are the required rational rank-one terms; zero columns may be omitted. Their polynomial encoding and construction time follow from Lemma~\ref{lem:factor-computation}. Rational matrix multiplication and positive-semidefinite tests have polynomial bit complexity, so the algorithm may also verify its output before returning it.
\end{proof}

The proof uses exact rational input and positive accuracy slack, together with standard polynomial-time univariate algebra. It is not a claim that a numerical eigensolver with a fixed tolerance finds the correct rank. The separate numerical-proposal prototype remains subject to its documented precision guard.

\section{The limit of the baseline scalar conversion}\label{app:scalar-limit}
This appendix identifies an equality case in Lemma~\ref{lem:conversion} when the branch padding is zero. The conclusion concerns that scalar argument alone, not the matrix-certificate class or approximation hardness.

\begin{proposition}[The scalar conversion alone cannot improve the factor]\label{prop:sharp}
For every $0<\tau\le1$, the zero-padding inequality
\[
 F\ge\alpha(\tau)\min\{L,F+\tau T\}
\]
for the single-edge scalar functions is attained at $x_i=x_j=\tau/(1+\tau)$.
\end{proposition}
\begin{proof}
Put $p=\tau/(1+\tau)\le1/2$. Then $F=2p-p^2$, $L=2p$, and $T=p(1-p)$. Direct substitution gives $F+\tau T=L$ and
\[
 \frac FL=1-\frac p2=\frac{2+\tau}{2+2\tau}.
\]
\end{proof}
This is a limitation of this scalar proof step. It is not an integrality-gap construction for the matrix-certificate class and not a hardness theorem. A stronger algorithm could use additional structure. We have not improved $5/6$ on the original $\gamma=1/2$ bipartite class.

\begingroup\small
\setlength{\parskip}{0pt}
\bibliographystyle{unsrt}
\bibliography{references}
\endgroup
\end{document}